\documentclass[12pt,a4paper]{scrartcl}
\pdfoutput=1
\usepackage[utf8]{inputenc}
\usepackage[T1]{fontenc}
\usepackage{lmodern}
\usepackage[english]{babel}
\usepackage{amsmath}
\usepackage{amssymb}
\usepackage{amsthm}
\usepackage{mathtools}
\usepackage{setspace}
\usepackage{mleftright}
\usepackage{nicefrac}
\usepackage[symbols,nogroupskip,sort=none]{glossaries-extra}
\usepackage[left=2.5cm,right=2.5cm,top=2.5cm,bottom=2.5cm]{geometry}
\usepackage{hyperref}
\hypersetup{breaklinks=true, hidelinks}
\numberwithin{equation}{section}
\theoremstyle{definition}
\newtheorem{defi}{Definition}[section]

\newtheorem{expl}[defi]{Example}
\newtheorem{nota}[defi]{Notation}
\theoremstyle{plain}
\newtheorem{trm}[defi]{Theorem}
\newtheorem{lem}[defi]{Lemma}
\newtheorem{cor}[defi]{Corollary}
\newtheorem{prop}[defi]{Proposition}
\newtheorem*{mtrm}{Main Theorem}
\newcommand{\p}[1]{\mleft(#1\mright)}
\newcommand{\cp}[1]{\mleft\{#1\mright\}}
\newcommand{\comp}[2]{{\normalfont(cf. \cite[#1]{#2})}}
\newcommand{\comptwo}[4]{{\normalfont(cf. \cite[#1]{#2}, \cite[#3]{#4})}}
\newcommand{\inN}{\in\mathbb{N}}
\newcommand{\seq}[2]{\p{#1_{#2}}_{#2\inN}}
\newcommand{\conv}[2]{#1\xrightarrow{\ \, \ }#2}
\begin{document}
\begin{center}
\textbf{\LARGE{On the Gromov-Hausdorff Limits of\\[0.2cm]Compact Surfaces with Boundary}}\\[0.5cm] 
\scriptsize{TOBIAS DOTT}
\end{center}
\begin{abstract}
In this work we investigate Gromov-Hausdorff limits of compact surfaces carrying length metrics. More precisely, we consider the case where all surfaces have the same Euler characteristic. We give a complete description of the limit spaces and study their topological properties. Our investigation builds on the results of a previous work which treats the case of closed surfaces.
\end{abstract}
\let\thefootnote\relax
\let\thefootnote\relax\footnotetext{The author was supported by the DFG grant SPP 2026 (LY 95/3-2).}
\section{Introduction}
Let $X$ be a simply connected compact ANR carrying a length metric and $M$ be a closed connected smooth manifold of dimension larger than two. From a result by Ferry and Okun it follows that $X$ can be obtained as the Gromov-Hausdorff limit of length spaces being homeomorphic to $M$ \comp{p. 1866}{FO95}. In dimension two this is not the case. For example a sequence of length spaces being homeomorphic to the 2-sphere can not converge to a space being homeomorphic to the 3-disc \comp{p. 269}{BBI01}.\\
This observation naturally leads to the following question: What do the Gromov-Hausdorff limits of length spaces being homeomorphic to a fixed closed surface look like? An answer was given by the author in \cite[pp. 13, 15]{Dot24}. In the present paper we completely describe the Gromov-Hausdorff limits of length spaces being homeomorphic to compact surfaces of fixed Euler characteristic. Our investigation particularly covers the case of non-vanishing boundary components.\\
It will turn out that the limit spaces satisfy the following topological properties:
\begin{trm}\label{trm_topo}
Let $X$ be a space that can be obtained as the Gromov-Hausdorff limit of length spaces being homeomorphic to a fixed compact surface. Then the following statements apply:
\begin{itemize}
\item[1)] $X$ is at most 2-dimensional. 
\item[2)] $X$ is locally simply connected.
\item[3)] There are finitely many compact surfaces $S_1,\ldots,S_n$ and $k\inN_0$ such that $\pi_1(X)$ is isomorphic to the free product $\pi_1\p{S_1}\ast\ldots\ast\pi_1\p{S_n}\ast\underbrace{\mathbb{Z}\ast\ldots\ast\mathbb{Z}}_{k\text{-times}}$.
\end{itemize}
\end{trm}
\noindent
Before we state our main result, we introduce some definitions: A subset $A$ of a Peano space $X$ is called $\emph{cyclicly connected}$ if every pair of points in $A$ can be connected by a simple closed curve in $A$. The subset is denoted as $\emph{maximal cyclic}$ provided it is not degenerate to a point, cyclicly connected and no proper subset of a cyclicly connected subset of $X$.\\
We have the following local description of the limit spaces:
\begin{trm}\label{trm_loc}
Let $X$ be a space that can be obtained as the Gromov-Hausdorff limit of length spaces being homeomorphic to a fixed compact surface. Then every point of $X$ admits an open neighborhood being homeomorphic to an open subset of some Peano space whose maximal cyclic subsets are homeomorphic to the 2-sphere or the 2-disc.   
\end{trm}
\noindent
Now we introduce further definitions for the global description: Let $X$ be a Peano space whose maximal cyclic subsets are compact surfaces and $C\subset X$ be a subcontinuum. Then $C$ is denoted as \emph{admissible} in $X$ provided $T\cap C$ is a point or a boundary component of $T$ for every maximal cyclic subset $T\subset X$. 
\begin{defi}\label{def_gen}
Let $X$ be a Peano space. Then $X$ is called a \emph{generalized cactoid} if the following statements apply:
\begin{itemize}
\item[1)] All maximal cyclic subsets are compact surfaces and only finitely many of them are not homeomorphic to the 2-sphere or the 2-disc.
\item[2)] There are finitely many disjoint admissible subcontinua $C_1,\ldots,C_n\subset X$ such that the boundary components of the maximal cyclic subsets of $X$ are covered by the subcontinua.
\end{itemize}
\end{defi}
\noindent  
It exists a natural choice $C_1,\ldots,C_n$ of the admissible subcontinua as above which is uniquely defined by the following property: The number $n$ is minimal and the union $\cup^n_{i=1}C_i$ is maximal among all choices with $n$ admissible subcontinua \textnormal{(}see Lemma \ref{lem_boundary}\textnormal{)}. We define the \emph{boundary} of $X$ as  $\cup^n_{i=1}C_i$ and denote it by  $\partial X$. Further we say that $C_i$ is a \emph{boundary component} of $X$.\\
Especially we will see that the boundary components are Peano spaces whose maximal cyclic subsets are homeomorphic to the 1-sphere (see Lemma \ref{lem_bound_cact}).\\A space being isometric to a metric quotient of $X$ whose underlying equivalence relation identifies exactly two points  is denoted as a \emph{metric 2-point identification} of $X$. If we consider a space that can be obtained by a successive application of $k>0$ metric 2-point identifications to some generalized cactoid $X$ and $p_1,\ldots,p_k$ denotes a choice of the corresponding projection maps, then $p_i$ is called a \emph{boundary identification} provided it identifies two points of $\p{p_{i-1}\circ\ldots\circ p_0}\p{\partial X}$ where $p_0\coloneqq id_X$.\\ 
The $\emph{connectivity number}$ of a compact surface $S$ is defined as $2-\chi(S)$. If we subtract the number of boundary components, we get the definition of the \emph{reduced connectivity number} of $S$. For a generalized cactoid we define the \emph{connectivity number} as the sum of the reduced connectivity numbers of its maximal cyclic subsets and the number of its boundary components. Moreover we say that a surface carrying a length metric is a \emph{length surface}.\\
The main result of this work completely describes the Gromov-Hausdorff closure of the class of compact length surfaces whose connectivity number is fixed: 
\begin{mtrm}\hypertarget{Main Theorem}
Let $c\inN_0$ and $X$ be a compact length space. Then the following statements are equivalent:
\begin{itemize}
\item[1)] $X$ can be obtained as the Gromov-Hausdorff limit of compact length surfaces whose connectivity number is equal to $c$.
\item[2)] There are $k,k_0\inN_0$ and a geodesic generalized cactoid $Y$ such that the following statements apply:
\begin{itemize}
\item[a)] $X$ can be obtained by a successive application of $k$ metric 2-point identifications to $Y$ such that $k_0$ of them are boundary identifications.
\item[b)] The connectivity number of $Y$ is less or equal to $c+k_0-2k$. 
\end{itemize}
\end{itemize}   
\end{mtrm}
\noindent
Furthermore we will investigate how the result changes if we restrict the first statement to orientable or non-orientable surfaces (see Theorem \ref{trm_main_first} and Theorem \ref{trm_main_sec}).\\ 
Since every compact length surface $S$ can be obtained as the limit of smooth Riemannian 2-manifolds being homeomorphic to $S$ and also of polyhedral surfaces being homeomorphic to $S$ \comptwo{p. 1674}{NR23}{p. 77}{RG93}, we deduce:
\begin{cor}
Let $c\inN_0$ and $X$ be a compact length space. Then the following statements are equivalent:
\begin{itemize}
\item[1)] $X$ can be obtained as the Gromov-Hausdorff limit of compact connected smooth Riemannian 2-manifolds whose connectivity number is equal to $c$.
\item[2)] $X$ can be obtained as the Gromov-Hausdorff limit of compact polyhedral surfaces whose connectivity number is equal to $c$.
\item[3)] There are $k,k_0\inN_0$ and a geodesic generalized cactoid $Y$ such that the following statements apply:
\begin{itemize}
\item[a)] $X$ can be obtained by a successive application of $k$ metric 2-point identifications to $Y$ such that $k_0$ of them are boundary identifications.
\item[b)] The connectivity number of $Y$ is less or equal to $c+k_0-2k$. 
\end{itemize}
\end{itemize}   
\end{cor}
\noindent
Beyond the main result of \cite{Dot24} we are only aware of the following predecessors: In the 1930's Whyburn described the limits of length spaces being homeomorphic to the 2-disc \textnormal{(}see Theorem \ref{trm_lim_disc}\textnormal{)}. Further Gromov states the first statement of Theorem \ref{trm_topo} for orientable surfaces without proof and attributes it to Ivanov \comp{p. 103}{Gro07}. From a  result by Carssola follows that every compact length space can be obtained as the limit of closed length surfaces \comp{p. 505}{Cas92}.\\
The latter result especially implies that the bound on the connectivity number is essential to our investigation.\\
This paper is organized as follows: In the preliminary notes we provide results on Gromov-Hausdorff convergence and the limits of closed length surfaces. Further we deal with the topology of Peano spaces, generalized cactoids and compact surfaces. In particular we show that the boundary of a generalized cactoid is well-defined and describe the topology of the boundary.\\
The aim of the third section is to show that the first statement of the \hyperlink{Main Theorem}{Main Theorem} implies the second. For this we start with a consideration of sequences with additional topological control. We also prove the first two statements of Theorem \ref{trm_topo} and give a proof of Theorem \ref{trm_loc}.\\ In Section \ref{sec_appr} we treat the remaining direction of the \hyperlink{Main Theorem}{Main Theorem}. At the end of the section we show the third statement of Theorem \ref{trm_topo}.\\
We note that the final results of Section \ref{sec_lim} and \ref{sec_appr} refine their corresponding statement of the \hyperlink{Main Theorem}{Main Theorem}. Together they describe how the \hyperlink{Main Theorem}{Main Theorem} changes if we restrict the first statement to orientable or non-orientable surfaces.   
\section{Preliminaries}
\subsection{Gromov-Hausdorff Convergence}
This subsection provides results on Gromov-Hausdorff convergence. Basic definitions and results regarding the Gromov-Hausdorff distance can be found in \cite[pp. 251-270]{BBI01}. A corresponding notion of convergence for maps is treated in \cite[pp. 401-402]{Pet16}.\\
For the sake of simplicity we note the following: If we consider a Gromov-Hausdorff convergent sequence, then there are isometric embeddings of the spaces and their limit into some compact metric space such that the induced sequence Hausdorff converges to the image of the limit \comp{pp. 64-65}{Gro81}. Whenever we apply this statement, we will identify corresponding sets without mentioning the underlying space.\\
Next we introduce the concept of almost isometries: If $f\colon X\to Y$ is a map between metric spaces, then we define its \emph{distortion} by:
\begin{align*}
dis(f)\coloneqq\sup_{x_1,x_2\in X}\cp{\left|d_Y\p{f\p{x_1},f\p{x_2}}-d_X\p{x_1,x_2}\right|}.\end{align*}
Further we call $f$ an $\varepsilon$-isometry provided $dis(f)\le\varepsilon$ and $f(X)$ is an $\varepsilon$-net in $Y$.\\
We have the following convergence criterion:
\begin{prop}\comp{p. 260}{BBI01}\label{prop_almost_isom}
Let $X$ be a compact metric space and $\seq{X}{n}$ be a sequence of compact metric spaces. Then the following statements are equivalent:
\begin{itemize}
\item[1)] The sequence converges to $X$.
\item[2)] For every $n\inN$ there is an $\varepsilon_n$-isometry $f_n\colon X_n\to X$ and $\conv{\varepsilon_n}{0}$. 
\end{itemize}
\end{prop}
\noindent
In particular the equivalence remains true if we interchange $X_n$ and $X$ in the second statement.\\
The property of being a length space is stable under Gromov-Hausdorff convergence:
\begin{prop}\comp{p. 265}{BBI01}
Let $X$ be a space that can be obtained as the Gromov-Hausdorff limit of compact length spaces. Then $X$ is a compact length space. \end{prop}
\noindent 
Finally we state Whyburn's theorem on the limits of discs:
\begin{trm}\comp{p. 422}{Why35}\label{trm_lim_disc}
Let $X$ be a space that can be obtained as the Gromov-Hausdorff limit of length spaces $\seq{X}{n}$ being homeomorphic to the 2-disc. Moreover we assume that $\seq{\partial X}{n}$ is convergent. Then $X$ is a compact length space satisfying the following properties:
\begin{itemize}
\item[1)] The maximal cyclic subsets of $X$ are homeomorphic to the 2-sphere or the 2-disc.
\item[2)] $X$ is a generalized cactoid with at most one boundary component.
\item[3)] The sequence $\seq{\partial X}{n}$ converges to a subset of $\partial X$.
\end{itemize}
\end{trm}
\subsection{Limits of Closed Surfaces}
As already mentioned the author dealt with the limits of closed length surfaces in a previous work. The aim of this subsection is to repeat some of the key results in \cite{Dot24}.\\ 
Our first result states two topological properties of the limit spaces. Throughout this work the term dimension refers to the covering dimension.
\begin{prop}\comp{p. 1}{Dot24}\label{prop_topo_closed}
Let $X$ be a space that can be obtained as the Gromov-Hausdorff limit of length spaces being homeomorphic to a fixed closed surface. Then the following statements apply:
\begin{itemize}
\item[1)] $X$ is at most 2-dimensional.
\item[2)] $X$ is locally simply connected.
\end{itemize}
\end{prop}
\noindent
The limit spaces can be described as follows:
\begin{trm}\comp{p. 13}{Dot24}\label{trm_lim_closed}
Let $X$ be a space that can be obtained as the Gromov-Hausdorff limit of closed length surfaces $\seq{X}{n}$ whose connectivity number is equal to $c$. Then $X$ can be obtained by a successive application of $k$ metric 2-point identifications to some geodesic generalized cactoid $Y$. Moreover the following statements apply:
\begin{itemize}
\item[1)]  The connectivity number of $Y$ is less or equal to $c-2k$ and its boundary is empty.
\item[2)] If $X_n$ is orientable for infinitely many $n\inN$, then the maximal cyclic subsets of $Y$ are orientable.
\item[3)] If $X_n$ is non-orientable for infinitely many $n\inN$ and the maximal cyclic subsets of $Y$ are orientable, then the connectivity number of $Y$ is less than $c$. 
\end{itemize}
\end{trm}
\noindent
Now we consider sequences with additional topological control: Throughout this work we call a simple closed curve a $\emph{Jordan curve}$. Let $\seq{X}{n}$ be a sequence of closed length surfaces. Then the sequence is denoted as \emph{regular} provided $\inf\cp{diam\p{J_n}\colon n\inN}$ is positive for every sequence $\seq{J}{n}$ such that $J_n$ is a non-contractible Jordan curve in $X_n$.\\
If we restrict the last theorem to regular sequences, then we derive the following result:
\begin{prop}\comp{p. 12}{Dot24}\label{prop_lim_reg_closed}
Let $X$ be a space that can be obtained as the Gromov-Hausdorff limit of closed length surfaces $\seq{X}{n}$ whose connectivity number is equal to $c>0$. If the sequence is regular, then $X$ is a compact length space satisfying the following property: All but one maximal cyclic subset are homeomorphic to the 2-sphere and one maximal cyclic subset is homeomorphic to $X_n$ for almost all $n\inN$.
\end{prop}
\subsection{Topology of Peano Spaces}
In this subsection we discuss results on the topology of Peano spaces. As main source about this topic we used \cite{Why42}.\\  
A compact connected metric space $X$ is called a \emph{continuum}. Further we denote a subset of $X$ being a continuum as a \emph{subcontinuum} of $X$. If $X$ is also locally connected, then we say that $X$ is a Peano space.\\
From a topological point of view Peano spaces and compact length spaces are the same:
\begin{prop}\comp{p. 1109}{Bin49}\label{prop_Peano_length}
Every Peano space is homeomorphic to a compact length space.   \end{prop}
\noindent
We present some basic properties of Peano spaces:
\begin{lem}\comp{pp. 65, 69, 71, 79}{Why42}\label{lem_Peano}
Let $X$ be a Peano space and $T\subset X$ be a maximal cyclic subset. Then the following statements apply:
\begin{itemize}
\item[1)] There are only countably many maximal cyclic subsets in $X$.
\item[2)] There are only countably many connected components in $X\setminus T$.
\item[3)] If $\seq{C}{n}$ is a sequence of pairwise distinct connected components of $X\setminus T$, then $\conv{diam\p{C_n}}{0}$. 
\item[4)] If $\seq{T}{n}$ is a sequence of pairwise distinct maximal cyclic subsets of $X$, then $\conv{diam\p{T_n}}{0}$. 
\item[5)] If $C$ is a connected component of $X\setminus T$, then there is some $x\in T$ such that $\partial C=\cp{x}$.
\item[6)] If $T_1,T_2\subset X$ are distinct maximal cyclic subsets, then $\left|T_1\cap T_2\right|\le 1$.
\end{itemize}
\end{lem}
\noindent
Let $X$ be a continuum. Then we say that $A\subset X$ separates $x,y\in X$ if the points lie in distinct connected components of $X\setminus A$. Moreover we call two distinct points of $X$ \emph{conjugate} to each other provided no point of $X$ separates them.\\
We state further properties of Peano spaces:
\begin{lem}\comp{pp. 65, 67, 79}{Why42}\label{lem_conj}
Let $X$ be a Peano space. Then the following statements apply:
\begin{itemize}
\item[1)] Every pair of conjugate points in $X$ can be connected by a Jordan curve in $X$.
\item[2)] Every non-degenerate cyclicly connected subset of $X$ lies in some maximal cyclic subset of $X$. 
\item[3)] If $T\subset X$ is a maximal cyclic subset and $A\subset T$ separates two points in $T$, then $A$ also separates these points in $X$.
\end{itemize}
\end{lem}
\noindent
Moreover we have the following characterization of cyclic connectedness:
\begin{prop}\comp{p. 79}{Why42}\label{prop_cycl_con_trm}
Let $X$ be a Peano space. Then $X$ is cyclicly connected if and only if there is no separating point in $X$.   
\end{prop}
\noindent
Provided all maximal cyclic subsets of a Peano space are homeomorpic to the 1-sphere, we denote the space as a \emph{1-cactoid}.\\ 
The following criterion for 1-cactoids by Whyburn will be a helpful tool: \begin{prop}\comp{p. 417}{Why35}\label{prop_1cact}
Let $X$ be a continuum. Then the following statements are equivalent:
\begin{itemize}
\item[1)] $X$ is a 1-cactoid.
\item[2)] For every pair of conjugate points $x,y\in X$ the subset $X\setminus\cp{x,y}$ is disconnected.
\end{itemize}
\end{prop}
\subsection{Generalized Cactoids}
In this subsection we deal with generalized cactoids.\\
First we describe the topology of admissible subcontinua:
\begin{lem}\label{lem_bound_cact}
Let $X$ be a Peano space whose maximal cyclic subsets are compact surfaces and $C\subset X$ be an admissible subcontinuum. Then $C$ is a 1-cactoid.
\end{lem}
\begin{proof}
Let $x,y\in C$ be conjugate. Then the points are also conjugate in $X$. Due to Lemma \ref{lem_conj} the points lie in some maximal cyclic subset $T\subset X$. Since $C$ is admissible, there is some boundary component $b\subset T$ such that $T\cap C=b$. Further we find an arc $\gamma\subset T$ satisfying the following property: The intersection of $\gamma$ with $b$ is given by $\cp{x,y}$ and $\gamma$ separates two points of $b$ in $T$.\\
By Lemma \ref{lem_conj} the arc also separates these points in $X$. It follows that $C\setminus\cp{x,y}$ is disconnected. From Proposition \ref{prop_1cact} we derive that $C$ is a 1-cactoid.
\end{proof}
\noindent
We have the following technical lemma:
\begin{lem}\label{lem_adm_stab}
Let $X$ be a Peano space whose maximal cyclic subsets are compact surfaces. Then the following statements apply:
\begin{itemize}
\item[1)] If $C_1,C_2\subset X$ are admissible subcontinua intersecting exactly once, then $C\coloneqq C_1\cup C_2$ is an admissible subcontinuum. 
\item[2)] Let $\seq{C}{n}$ be a sequence of subcontinua in $X$ which intersect every maximal cyclic subset at most once. If the sequence Hausdorff converges to some $C\subset X$, then $C$ is a subcontinuum which intersects every maximal cyclic subset at most once. \end{itemize}
\end{lem}
\begin{proof}
1) It directly follows that $C$ is a continuum. Let $T\subset X$ be a maximal cyclic subset and $x,y\in T\cap C$ be distinct points. Then we may assume that $x\in C_1\setminus C_2$.\\
For the sake of contradiction we further assume that $y\in C_2\setminus C_1$. By Lemma \ref{lem_bound_cact} the subsets $C_1$ and $C_2$ are arcwise connected. Hence there is an arc $\gamma\subset C$ connecting $x$ and $y$ which is the union of non-degenerate arcs $\gamma_1\subset C_1$ and $\gamma_2\subset C_2$. From Lemma \ref{lem_Peano} we get $\gamma\subset T$. Because $C_1$ and $C_2$ are admissible, there are boundary component $b_1,b_2\subset T$ with $T\cap C_1=b_1$ and $T\cap C_2=b_2$. Since $\gamma$ is connected, we get $b_1=b_2$. This contradicts the fact that $\left|C_1\cap C_2\right|=1$.\\
We conclude that $T\cap C=T\cap C_1$. Due to the fact that $C_1$ is admissible, it follows that $T\cap C$ is a boundary component of $T$. Therefore we derive that $C$ is admissible.\\
2) We directly get that $C$ is a continuum. For the sake of contradiction we assume the existence of a maximal cyclic subset $T\subset X$ and distinct points $x_1,x_2\in T\cap C$. Then there is a sequence $\p{x_{i,n}}_{n\inN}$ converging to $x_i$ such that $x_{i,n}\in C_n$ and $x_{1,n}\neq x_{2,n}$. By Proposition \ref{prop_Peano_length} we may assume that $X$ is geodesic and we find a geodesic $\gamma_{i,n}\subset X$ connecting $x_{i,n}$ and $x_i$. Since $x_1\neq x_2$, we may assume that the geodesics do not intersect. From Lemma \ref{lem_bound_cact} we derive that $C_n$ is arcwise connected. Hence there is a non-degenerate arc $\alpha_n\subset C_n$ connecting $x_{1,n}$ and $x_{2,n}$. Due to the fact that $\gamma_{1,n}$ and $\gamma_{2,n}$ do not intersect, we may assume that $\alpha_n$ intersects $\gamma_{1,n}\cup\gamma_{2,n}$ only twice. Then $\gamma_n\coloneqq \gamma_{1,n}\cup\alpha_n\cup\gamma_{2,n}$ is an arc and Lemma \ref{lem_Peano} yields $\gamma_n\subset T$. This contradicts the fact that $C_n$ intersects every maximal cyclic subset at most once.
\end{proof}
\noindent
Next we introduce some definitions: Let $X$ be a Peano space whose maximal cyclic subsets are compact surfaces. Further let $C_1,\ldots,C_n\subset X$ be disjoint admissible subcontinua such that the boundary components of the maximal cyclic subsets are covered by the subcontinua. If every $C_i$ contains a boundary component of some maximal cyclic subset, then we denote $\cup_{i=1}^{n} C_i$ as a \emph{pre-boundary} of $X$. Provided the number $n$ is minimal among all pre-boundaries of $X$, we say that the pre-boundary is \emph{minimal}.\\
We note that the second property of Definition \ref{def_gen} can be restated as the existence of a pre-boundary. The following example illustrates this property:
\begin{expl}
We consider a subset $X\subset\mathbb{R}^3$ which is the union of
the 2-sphere and disjoint subsets $\seq{D}{n}$ being homeomorphic to the 2-disc. Further we assume that $\conv{diam\p{D_n}}{0}$ and that $D_n$ intersects the 2-sphere exactly once for every $n\inN$. Then $X$ is a Peano space whose maximal cyclic subsets are homeomorphic to the 2-sphere or the 2-disc. Hence the first property of Definition \ref{def_gen} is satisfied. But $X$ is not a generalized cactoid since there is no pre-boundary in $X$.\\ 
For our second example we replace the 2-sphere in the construction above with a further subset $D\subset \mathbb{R}^3$ being homeomorphic to the 2-disc. We denote this new subset by $Y$. Then $Y$ is a Peano space whose maximal cyclic subsets are homeomorphic to the 2-disc. Hence the first property of Definition \ref{def_gen} is again satisfied. Moreover $Y$ is a generalized cactoid if and only if $\partial D_n$ intersects $\partial D$ for almost all $n\inN$.
\end{expl}
\noindent
Finally we show that the boundary of a generalized cactoid is well-defined:  
\begin{lem}\label{lem_boundary}
Let $X$ be a generalized cactoid. Then the union of all minimal pre-boundaries is a minimal pre-boundary.
\end{lem}
\begin{proof}
Let $P\subset X$ be a minimal pre-boundary. We denote the set of all subcontinua in $X$ which intersect every maximal cyclic subset at most once by $\mathcal{T}$. Further we define $P_0$ as the union of $P$ with all subcontinua in $\mathcal{T}$ which intersect $P$ exactly once.\\
We show that $P_0$ is a minimal pre-boundary: Since $P$ covers the boundary components of the maximal cyclic subsets, the same applies to $P_0$. Moreover $P_0$ has at most as many connected components as $P$.\\ We prove that every connected component $C\subset P_0$ is compact: Let $\seq{x}{n}$ be a sequence in $P_0$ and $x\in X$ with $\conv{x_n}{x}$. By construction there is a subcontinuum $A_n\in\mathcal{T}$ which contains $x_n$ and intersects $P$ exactly once. After passing to a subsequence, we may assume that the subcontinua $\seq{A}{n}$ Hausdorff converge to some subcontinuum $A\subset X$.\\
Since $P$ is compact, the intersection of $A$ and $P$ is non-empty. From Lemma \ref{lem_adm_stab} it follows that $A$ intersects every maximal cyclic subset at most once. In particular $A$ is admissible and hence arcwise connected by Lemma \ref{lem_bound_cact}. Therefore we find an arc $\gamma\subset A$ connecting $x$ and $P$ which intersects $P$ exactly once. Especially we have $\gamma\in\mathcal{T}$ and we deduce $\gamma\subset P_0$. This yields $x\in P_0$.\\
We conclude that $P_0$ is closed. Because $X$ is compact, we get that $P_0$ is compact. Hence $C$ is also compact.\\ 
Due to the fact that the subcontinua in $\mathcal{T}$ and the connected components of $P$ are admissible, Lemma \ref{lem_adm_stab} yields that $C$ is admissible. We conclude that $P_0$ is a minimal pre-boundary.\\ 
Next we show that every pre-boundary $A\subset X$ lies in $P_0$: For the sake of contradiction we assume that $A$ is not a subset of $P_0$. Then we find some $p\in A\setminus P_0$. We denote the connected component of $A$ containing $p$ by $C$. The subset $C$ contains a boundary component of some maximal cyclic subset. Moreover $P$ covers the boundary components of the maximal cyclic subsets. Therefore $C$ intersects $P$. Since $C$ is admissible, it is arcwise connected. Hence there is an arc $\gamma\subset C$ starting in $p$ whose endpoint $e$ is the only intersection point of $\gamma$ with $P$.\\
We note that $\gamma\setminus\cp{e}$ does not intersect boundary components of maximal cyclic subsets. Due to the fact that $C$ is admissible, we have $\gamma\in\mathcal{T}$. Finally we deduce $\gamma\subset P_0$ and therefore $p\in P_0$. A contradiction.
\end{proof}
\noindent
We remark that the boundary components of a generalized cactoid are 1-cactoids by Lemma \ref{lem_bound_cact}.
\subsection{Curves in Compact Surfaces}
This subsection is devoted to the classification of curves in compact surfaces.\\
Let $\gamma$ be an arc in a compact surface $S$. Then $\gamma$ is called \emph{simple} if its endpoints lie on boundary components and the interior of $\gamma$ does not. Moreover we denote $\gamma$ as \emph{separating} provided $S\setminus \gamma$ is disconnected.\\
The next two results yield a classification of simple arcs in compact surfaces:  
\begin{prop}\comp{pp. 54-55}{MST16}\label{prop_clas_arc}
Let $S$ be a compact surface of connectivity number $c$. Further let $\gamma\subset S$ be a separating  simple arc which does not form a contractible Jordan curve together with a subarc of some boundary component.\\
Then there are $c_1,c_2\in\mathbb{N}_{\ge 2}$ with $c_1+c_2=c+1$ and a compact surface $S_i$ of connectivity number $c_i$ such that the topological quotient $S/\gamma$ is a wedge sum of $S_1$ and $S_2$. Moreover the wedge point lies in $\partial S_1\cap\partial S_2$.\\
At least one of the surfaces is non-orientable if and only if $S$ is non-orientable.
\end{prop}
\begin{prop}\comp{pp. 54-55}{MST16}\label{prop_clas_arc_II}
Let $S$ be a compact surface of connectivity number $c$ and $\gamma\subset S$ be a non-separating simple arc.\\
Then there is a compact surface $S_1$ of connectivity number $c-1$ such that $S/\gamma$ is a topological 2-point identification of $S_1$. Moreover the glued points lie in $\partial S_1$.\\
If $S$ is orientable, then $S_1$ is orientable. 
\end{prop}
\noindent
We say that a Jordan curve $J\subset S$ is \emph{simple} provided we have $\left|J\cap\partial S\right|\le 1$.\\ 
There is also a classification of non-contractible simple Jordan curves in compact surfaces:
\begin{prop}\comp{pp. 54-55}{MST16}\label{prop_clas_Jord}
Let $S$ be a compact surface of connectivity number $c$ and $J\subset S$ be a non-contractible simple Jordan curve. Then the topological quotient $X\coloneqq S/J$ can be described in one of the following ways: 
\begin{itemize}
\item[1)] There are $c_1,c_2\in\mathbb{N}$ with $c_1+c_2=c$ and a compact surface $S_i$ of connectivity number $c_i$ such that $X$ is a wedge sum of $S_1$ and $S_2$. Moreover at least one of the surfaces is non-orientable if and only if $S$ is non-orientable. 
\item[2)] There is a compact surface of connectivity number $c-2$ such that $X$ is a topological 2-point identification of it. Moreover the surface is orientable if $S$ is orientable.
\item[3)] $X$ is a compact surface of connectivity number $c-1$ and $S$ is non-orientable. 
\end{itemize}
\end{prop}
\subsection{Fundamental Group Formulas}
We provide two fundamental group formulas. First we consider locally simply connected Peano spaces. In a previous work the author showed the following result:
\begin{prop}\comp{p. 10}{Dot24}\label{prop_funda_form}
Let $X$ be a locally simply connected Peano space and $\seq{T}{n}$ be an enumeration of its maximal cyclic subsets. Then $\pi_1(X)$ is isomorphic to $\pi_1\p{T_1}\ast\ldots\ast\pi_1\p{T_n}$ for almost all $n\inN$.     
\end{prop}
\noindent
Further we consider topological 2-point identifications. The next proposition is a consequence of the HNN-Seifert-van Kampen Theorem in \cite[p. 1435]{Fri23}:
\begin{prop}\label{prop_funda_form_II}
Let $X$ be a locally simply connected and path-connected topological space. Further let $Y$ be a topological 2-point identification of $X$. Then $\pi_1(Y)$ is isomorphic to $\pi_1(X)\ast\mathbb{Z}$. 
\end{prop}
\glsxtrnewsymbol[description={The class of compact metric spaces.}]{1}{$\mathcal{M}$}
\glsxtrnewsymbol[description={The class of compact length surfaces whose connectivity is equal to $c$.}]{2}{$\mathcal{S}(c)$}
\glsxtrnewsymbol[description={The class of compact length surfaces with $b$ boundary components whose reduced connectivity number is equal to $c$.}]{3}{$\mathcal{S}\p{c,b}$}
\glsxtrnewsymbol[description={The class of geodesic generalized cactoids whose connectivity number is equal $c$.}]{
4}{$\mathcal{G}(c)$}
\glsxtrnewsymbol[description={The class of geodesic generalized cactoids with $b$ boundary components such that the reduced connectivity numbers of their maximal cyclic subsets sum up to $c$.}]{4}{$\mathcal{G}\p{c,b}$}
\glsxtrnewsymbol[description={The class of successive metric wedge sums of non-degenerate cyclicly connected compact length spaces and finite metric trees.}]{6}{$\mathcal{W}$}
\glsxtrnewsymbol[description={The class of successive metric wedge sums of compact length surfaces such that every wedge point is only shared by two of their surfaces.}]{7}{$\mathcal{W}_0$}
\printunsrtglossary[type=symbols,style=long,title=Notation]
\noindent
We note that we allow a change of the wedge point in every construction step of a successive metric wedge sum. 
\section{The Limit Spaces}\label{sec_lim}
The goal of this chapter is to show that the first statement of the \hyperlink{Main Theorem}{Main Theorem} implies the second.
\subsection{Topological Properties}\label{subsec_topo}
First we prove that the limit spaces are at most 2-dimensional and locally simply connected. For this we introduce some notations: 
\begin{nota}\label{nota_doub}
Let $\seq{X}{n}$ be a sequence in $\mathcal{S}\p{c,b}$, where $b>0$, and $X\in\mathcal{M}$ with $\conv{X_n}{X}$.\\
We denote the metric gluing of $X_n\sqcup X_n$ along $\partial X_n$ by $2X_n$. Then the sequence $\seq{2X}{n}$ is convergent and we denote its limit by $2X$. Especially there are subsets $X^\pm\subset 2X$ and maps $\tau^\pm\colon 2X\to X^\pm$ such that the following statements apply:
\begin{itemize}
\item[1)] $X^+\cup X^-=2X$.
\item[2)] $X^\pm$ is isometric to $X$.
\item[3)] The restriction of $\tau^\pm$ to $X^\pm$ is the identity map and the restriction to $X^\mp$ is an isometry.
\item[4)] The restriction of $\tau^\pm\circ\tau^\mp$ to $X^\pm$ is the identity map. 
\end{itemize}
We fix some isometries as in the second statement. For every $A\subset X$ we denote its corresponding subset of $X^\pm$ by $A^\pm$. After passing to a subsequence, we may and will assume that $\seq{\partial X}{n}$ is convergent and we denote its limit by $\partial^\infty X$. Finally also the following property is satisfied: 
\begin{itemize}
\item[5)] $X^+\cap X^-=\p{\partial^\infty X}^+=\p{\partial^\infty X}^-$. 
\end{itemize}
\end{nota}
\noindent
Now we show that the limit spaces fulfill the two topological properties:
\begin{proof}[Proof of Theorem \ref{trm_topo} (Part I)] There is a sequence $\seq{X}{n}$ of length spaces being homeomorphic to a fixed compact surface such that $\conv{X_n}{X}$. If $X_n$ is a closed surface, then $X$ is at most 2-dimensional and locally simply connected by Proposition \ref{prop_topo_closed}. Hence we may assume that the boundary of $X_n$ is non-empty.\\
1) We note that $2X_n$ is a closed surface. Hence $2X$ is at most 2-dimensional. Since we have $X^+\subset 2X$, we derive that $X^+$ is also at most 2-dimensional \comp{p. 266}{Sak13}. Because $X$ and $X^+$ are isometric, we deduce that $X$ is at most 2-dimensional.\\
2) The statements listed in Notation \ref{nota_doub} imply the following: Let $V\subset 2X$ be an open and simply connected subset. Since $V$ is open, the same applies to $\tau^+(V)$. Moreover every loop in $\tau^+(V)$ is the composition of $\tau^+$ with some loop in $V$. The map $\tau^+$ is continuous. Due to the fact that $V$ is simply connected, we hence get that $\tau^+(V)$ is simply connected. We further note that the diameter of $\tau^+(V)$ does not exceed that of $V$.\\
Because $2X_n$ is a closed surface, the space $2X$ is locally simply connected. It follows that $X^+$ is locally simply connected. Therefore the same applies to $X$.
\end{proof}
\subsection{Regular Convergence}
In this subsection we consider sequences in $\mathcal{S}\p{c,b}$ with additional topological control: 
\begin{defi}
Let $\seq{X}{n}$ be a sequence in $\mathcal{S}\p{c,b}$ where $b>0$. Then the sequence is called \emph{regular} provided $\inf\cp{diam\p{J_n}\colon n\inN}$ is positive for every sequence $\seq{J}{n}$ such that $J_n$ is a non-contractible Jordan curve in $2X_n$.
\end{defi} 
\noindent
In other words, the sequence $\seq{X}{n}$ is regular if and only if the sequence $\seq{2X}{n}$ is regular. Provided the sequence $\seq{X}{n}$ converges to some $X\in\mathcal{M}$, the definition directly implies that $\partial^\infty X$ has $b$ connected components.\\
The next result states a property of non-regular sequences. We remark that Notation \ref{nota_doub} can also be applied to the sequence constantly being $X_n$. In the upcoming proof we use this notation and denote the corresponding maps by $\tau^\pm_n$.   
\begin{lem}\label{lem_non_reg}
Let $\seq{X}{n}$ be a non-regular sequence in $\mathcal{S}\p{c,b}$ where $b>0$. After passing to a subsequence, we may assume that one of the following cases applies: 
\begin{itemize}
\item[1)] There is a sequence $\seq{\gamma}{n}$ such that $\gamma_n$ is a non-separating simple arc in $X_n$. Moreover we have $\conv{diam\p{\gamma_n}}{0}$.
\item[2)] There is a sequence $\seq{\gamma}{n}$ such that $\gamma_n$ is a separating simple arc in $X_n$ which does not form a contractible Jordan curve together with a subarc of some boundary component. Moreover we have $\conv{diam\p{\gamma_n}}{0}$. 
\item[3)] There is a sequence $\seq{\gamma}{n}$ such that $\gamma_n$ is a non-contractible simple Jordan curve in $X_n$. Moreover we have $\conv{diam\p{\gamma_n}}{0}$.    
\end{itemize}
\end{lem}
\begin{proof}
We consider the case that the first two statements of the lemma do not apply. By non-regularity we may assume the existence of a sequence $\seq{J}{n}$ such that $J_n$ is a non-contractible Jordan curve in $2X_n$ and $\conv{diam\p{J_n}}{0}$. Since the first statement of the lemma does not apply, we may assume that $J_n$ intersects exactly one boundary component $b_n\subset X_n^+$ and the intersection is non-degenerate.\\ There is a homeomorphism $f_n\colon 2X_n\to Y_n$ such that $Y_n$ is a Riemannian manifold and $f_n\p{b_n}$ is a piecewise geodesic Jordan curve. We note that the Jordan curve $f_n\p{J_n}$ can be obtained as the Hausdorff limit of piecewise geodesic Jordan curves being homotopic to $f_n\p{J_n}$ \comptwo{p. 1794}{Shi99}{pp. 413-415}{Why35}. Since $Y_n$ is a compact Riemannian manifold, there is some $\varepsilon_n>0$ such that every pair of points in $Y_n$ with distance less than $\varepsilon_n$ can be connected by a unique geodesic. It follows that two distinct geodesic Jordan curves in $Y_n$ intersect at most finitely many times.\\
By the observations above we may assume that $J_n$ and $b_n$ intersect only finitely many times. Then there is a finite subdivision of $J_n$ into simple arcs in $X_n^+$ and $X_n^-$ and arcs in $b_n$. Because the second statement of the lemma does not apply, we may assume that every of the arcs is homotopic to some arc in $b_n$. This yields that $J_n$ is homotopic to some loop in $b_n$.\\
We derive that $\gamma_n\coloneqq \tau_n^+\p{J_n}$ is a non-contractible loop in $X_n^+$ and $\conv{diam\p{\gamma_n}}{0}$. In particular we may assume that $\gamma_n$ does not intersect $\partial X_n^+$. Moreover we find a Jordan curve in $\gamma_n$ which is non-contractible in $X_n^+$ \comp{p. 626}{LW18}. Therefore we finally may assume that $\gamma_n$ is a Jordan curve. 
\end{proof}
\subsubsection{Limits of Boundary Components}
As a first step we  investigate the limits of boundary components for regular sequences. In the next three results we extend Whyburn's proof ideas regarding the limits of discs \comp{pp. 421-424}{Why35}: 
\begin{prop}\label{prop_lim_cact}
Let $\seq{X}{n}$ be a regular sequence in $\mathcal{S}\p{c,q}$, where $q>0$, and $X\in\mathcal{M}$ with $\conv{X_n}{X}$. If $b$ is a connected component of $\partial^\infty X$, then $b$ is a 1-cactoid. 
\end{prop}
\begin{proof}
By regularity there is a sequence $\seq{b}{n}$ such that $b_n$ is a boundary component of $X_n$ and $\conv{b_n}{b}$. Since $b$ can be obtained as the Hausdorff limit of continua, it is also a continuum.\\
For the sake of contradiction we assume that $b$ is not a 1-cactoid. From Proposition \ref{prop_1cact} it follows the existence of conjugate points $x,y\in b$ such that $b\setminus\left\{x,y\right\}$ is connected. There are sequences $\seq{x}{n}$ and $\seq{y}{n}$ with $x_n,y_n\in b_n$ and $x_n\neq y_n$ such that $\conv{x_n}{x}$ and $\conv{y_n}{y}$. We denote the subarcs of $b_n$ connecting $x_n$ and $y_n$ by $\alpha_n$ and $\beta_n$. Moreover we may assume that there are $\alpha,\beta\subset b$ such that $\conv{\alpha_n}{\alpha}$ and $\conv{\beta_n}{\beta}$.\\
Then we have $\alpha\cup\beta=b$ and it exists $z\in \alpha\cap \beta\setminus\cp{x,y}$. Further we choose sequences $\seq{z}{n}$ and $\seq{\Tilde{z}}{n}$ with $z_n\in \alpha_n$ and $\Tilde{z}_n\in \beta_n$ such that $\conv{z_n}{z}$ and $\conv{\Tilde{z}_n}{z}$.\\
Let $\gamma_n\subset X_n$ be a geodesic between $z_n$ and $\Tilde{z}_n$. After passing to a subsequence and subarcs of the geodesics, we may assume $\gamma_n$ to be a simple arc. By regularity and $\conv{diam\p{\gamma_n}}{0}$ we also may assume that $\gamma_n$ is separating.\\
Now there are compact surfaces $U_n,V_n\subset X_n$ such that $x_n\in U_n$, $y_n\in V_n$, $U_n\cup V_n=X_n$ and $U_n\cap V_n=\gamma_n$. We may assume the corresponding sequences to be convergent with limits $U$ and $V$. This leads to $U\cup V=X$ and $U\cap V=\cp{z}$ \comp{p. 412}{Why35}. Finally we derive that $z$ separates $x$ and $y$ in $X$ and therefore also in $b$. A contradiction.
\end{proof}
\noindent
The proof above also demonstrates the following lemma:
\begin{lem}\label{lem_sep}
Let $\seq{X}{n}$ be a regular sequence in $\mathcal{S}\p{c,q}$, where $q>0$, and $X\in\mathcal{M}$ with $\conv{X_n}{X}$. If $b$ is a connected component of $\partial^\infty X$ and $x,y,z\in b$ are such that $z$ separates $x$ and $y$ in $b$, then $z$ separates $x$ and $y$ in $X$. 
\end{lem}
\noindent
In the next two results we study the intersections with maximal cyclic subsets of $2X$:
\begin{lem}\label{lem_b_hom_S}
Let $\seq{X}{n}$ be a regular sequence in $\mathcal{S}\p{c,q}$, where $q>0$, and $X\in\mathcal{M}$ with $\conv{X_n}{X}$. If $b$ is a connected component of $\partial^\infty X$ and $T$ is a maximal cyclic subset of $2X$ with $\left|T\cap b^+\right|>1$, then the intersection is homeomorphic to the 1-sphere. 
\end{lem}
\begin{proof}
Let $x,y\in T\cap b^+$ with $x\neq y$. Since $x,y\in T$, the points lie on some Jordan curve in $2X$. If $\gamma\subset 2X$ is a path between $x$ and $y$ which does not contain a certain point of $b^+$, then $\tau^+\circ\gamma$ is also such a path. Therefore Lemma \ref{lem_sep} implies that $x$ and $y$ are conjugate in $b^+$.\\
By Proposition \ref{prop_lim_cact} the subset $b^+$ is a 1-cactoid. From Lemma \ref{lem_conj} we derive that $x$ and $y$ are contained in some maximal cyclic subset $S\subset b^+$. We note that $S$ is homeomorphic to the 1-sphere. Using Lemma \ref{lem_Peano} and Lemma \ref{lem_conj}, we get $T\cap b^+\subset S$ and $S\subset T$. Especially we have $S\subset T\cap b^+$ and therefore $S=T\cap b^+$.
\end{proof} 
\begin{lem}\label{lem_con_inter}
Let $\seq{X}{n}$ be a regular sequence in $\mathcal{S}\p{c,q}$, where $q>1$, and $X\in\mathcal{M}$ with $\conv{X_n}{X}$. If $b$ is a connected component of $\partial^\infty X$, then there is a maximal cyclic subset of $2X$ which intersects  $b^+$ and a further connected component of $\p{\partial^\infty X}^+$.
\end{lem}
\begin{proof}
By regularity there is a sequence $\seq{b}{n}$ such that $b_n$ is a boundary component of $X_n$ and $\conv{b_n}{b}$. We choose a geodesic $\gamma_n\subset X_n$ between some point of $b_n$ and some point of $\partial X_n\setminus b_n$. Then we may assume the existence of a geodesic $\gamma\subset X$ such that $\conv{\gamma_n}{\gamma}$.\\
Due to regularity $\gamma$ connects some point of $b$ with some point of $\partial^\infty X\setminus b$. After passing to a subarc, we may assume that the interior of $\gamma$ does not intersect $\partial^\infty X$. It follows that $J\coloneqq\gamma^+\cup\gamma^-$ is a non-degenerate Jordan curve in $2X$. By Lemma \ref{lem_conj} there is a maximal cyclic subset $T\subset 2X$ containing $J$. We conclude that $T$ intersects $b^+$ and a further connected component of $\p{\partial^\infty X}^+$. 
\end{proof}
\subsubsection{Regular Limit Spaces}
Now we describe the limits of regular sequences:
\begin{lem}\label{lem_reg_max}
Let $\seq{X}{n}$ be a regular sequence in $\mathcal{S}\p{c,b}$, where $b>0$, and $X\in\mathcal{M}$ with $\conv{X_n}{X}$. Further let $T$ be a maximal cyclic subset of $X^+$. Then one of the following cases applies: 
\begin{itemize}
\item[1)] $T$ is a maximal cyclic subset of $2X$ and is a closed surface. Moreover we have $\left|T\cap\p{\partial^\infty X}^+\right|\le 1$.
\item[2)] $T\cup\tau^-(T)$ is a maximal cyclic subset of $2X$ and $T$ is a compact surface with non-empty boundary. Moreover we have $\partial T=T\cap\p{\partial^\infty X}^+$. 
\end{itemize}
Also the following statement applies: Every maximal cyclic subset of $2X$ which is not a maximal cyclic subset of $X^+$ or $X^-$ can be obtained as in the second case.
\begin{proof}
Since the sequence is regular, Proposition \ref{prop_lim_reg_closed} implies that every maximal cyclic subset of $2X$ is a closed surface.\\
First we consider the case that $T$ is a maximal cyclic subset of $2X$: For the sake of contradiction we assume $\left|T\cap\p{\partial^\infty X}^+\right|>1$. Then Proposition \ref{prop_cycl_con_trm} implies that $T\cup\tau^-(T)$ is cyclicly connected. Further we have $T\neq\tau^-(T)$ by Proposition \ref{prop_lim_cact}. Hence $T$ is a proper subset of some cyclicly connected subset in $2X$. A contradiction.\\
Now we consider the case that $T$ is not a maximal cyclic subset of $2X$: By Lemma \ref{lem_conj} there is a maximal cyclic subset $S\subset 2X$ containing $T$.\\
Let $V$ be the closure of a connected component of $S\setminus\p{\partial^\infty X}^+$. Then we may assume $V\subset X^+$. As a consequence of Lemma \ref{lem_b_hom_S}, the subset $V$ is a compact surface. It also follows that $V\cap\p{\partial^\infty X}^+$ is a disjoint union of $\partial V$ and $k$ points. In particular $\partial V$ is non-empty since $S$ is a closed surface. Due to the fact that $W\coloneqq V\cup\tau^-(V)$ is cyclicly connected, Lemma \ref{lem_Peano} and Lemma \ref{lem_conj} yield $W\subset S$. Because $S$ is a closed surface, we derive $k=0$ and $S=W$. This implies $V=X^+\cap S$ and therefore $T\subset V$. We note that $V$ is cyclicly connected. Hence we get $T=V$.\\
Using Lemma \ref{lem_Peano} and Lemma \ref{lem_conj}, the paragraph above also implies the last statement of our result. 
\end{proof}
\end{lem}
\noindent
It follows the main result of this subsection:
\begin{prop}\label{prop_reg_lim}
Let $\seq{X}{n}$ be a regular sequence in $\mathcal{S}\p{c,b}$ where $b>0$ and $c+b>1$. Further let $X\in\mathcal{M}$ with $\conv{X_n}{X}$. Then $X$ is a compact length space satisfying the following properties:
\begin{itemize}
\item[1)] All but one maximal cyclic subset are homeomorphic to the 2-sphere or the 2-disc and one maximal cyclic subset is homeomorphic to $X_n$ for almost all $n\inN$.
\item[2)] $X$ is a generalized cactoid with $b$ boundary components and $\partial^\infty X\subset \partial X$. 
\end{itemize}
\end{prop}
\begin{proof} From Proposition \ref{prop_lim_reg_closed} and Lemma \ref{lem_reg_max} we get the following: All but one maximal cyclic subset $T\subset X$ are homeomorphic to the 2-sphere or the 2-disc. Moreover $T$ is a compact surface with non-empty boundary whose connectivity number is equal to $c+b$. We also have that $T$ is orientable if and only if $X_n$ is orientable for almost all $n\inN$.\\
By regularity $\partial^\infty X$ has $b$ connected components. Combining Lemma \ref{lem_con_inter} and Lemma \ref{lem_reg_max}, we derive that $T$ has $b$ boundary components. Hence the reduced connectivity number of $T$ is equal to $c$. This yields that $T$ is homeomorphic to $X_n$ for almost all $n\inN$.\\
Moreover the connected components of $\partial^\infty X$ are disjoint subcontinua of $X$. Due to Lemma \ref{lem_b_hom_S} and Lemma \ref{lem_reg_max} the subcontinua are admissible and they cover the maximal cyclic subsets of $X$. Therefore $X$ is a generalized cactoid.\\
Since $T$ has $b$ boundary components, the pre-boundary $\partial^\infty X$ is minimal. We conclude that $X$ has $b$ boundary components and $\partial^\infty X\subset \partial X$.
\end{proof}
\subsection{The General Case}
\noindent
The aim of this subsection is to prove that the first statement of the \hyperlink{Main Theorem}{Main Theorem} implies the second.\\
We already described the limits of closed length surfaces and regular sequences. Now we investigate non-regular sequences: Let $\seq{X}{n}$ be a non-regular sequence in $\mathcal{S}\p{c,b}$ where $b>0$.  Further let $X\in\mathcal{M}$ with $\conv{X_n}{X}$.\\
By non-regularity we may assume that there is a sequence $\seq{\gamma}{n}$ of simple arcs or simple Jordan curves as in Lemma \ref{lem_non_reg}. Since the diameters of the curves vanish, the metric quotients $X_n/\gamma_n$ converge to $X$. From Proposition \ref{prop_clas_arc}, Proposition \ref{prop_clas_arc_II} and Proposition \ref{prop_clas_Jord} we get a topological description of these quotient spaces. Using this description, we derive the following two results:
\begin{lem}\label{lem_arcs} If the curves of the sequence $\seq{\gamma}{n}$ are simple arcs, then one of the following cases applies:
\begin{itemize}
\item[1)] There are $c_1,c_2\in\mathbb{N}_{\ge 2}$ with $c_1+c_2=c+1$ and a sequence $\p{Y_{i,n}}_{n\inN}$ in $\mathcal{S}\p{c_i}$ converging to some $Y_i\in\mathcal{M}$ such that $X$ is isometric to a metric wedge sum of $Y_1$ and $Y_2$. Furthermore the wedge point lies in $\partial^\infty Y_1\cap \partial^\infty Y_2$ and we find a corresponding isometry $p$ such that $p\p{\partial^\infty Y_1\cup \partial^\infty Y_2}=\partial^\infty X$.\\ Provided $X_n$ is non-orientable for infinitely many $n\inN$, the surfaces of at least one of the sequences may be chosen to be non-orientable.
\item[2)] There is a sequence $\seq{Y}{n}$ in $\mathcal{S}\p{c-1}$ converging to some $Y\in \mathcal{M}$ such that $X$ is isometric to $Y$ or a metric 2-point identification of it.  Furthermore the glued points lie in $\partial^\infty Y$ and we find a corresponding isometry or projection map $p$ such that $p\p{\partial^\infty Y}=\partial^\infty X$. 
\end{itemize}
If $X_n$ is orientable for infinitely many $n\inN$, then the surfaces of the sequences above may be chosen to be orientable.
\end{lem}      
\begin{lem}\label{lem_Jor}
If the curves of the sequence $\seq{\gamma}{n}$ are simple Jordan curves, then one of the following cases applies:
\begin{itemize}
\item[1)] There are $c_1,c_2\inN$ with $c_1+c_2=c$ and a sequence $\p{Y_{i,n}}_{n\inN}$ in $\mathcal{S}\p{c_i}$ converging to some $Y_i\in\mathcal{M}$ such that $X$ is isometric to a metric wedge sum of $Y_1$ and $Y_2$. Furthermore we find a corresponding isometry $p$ such that $p\p{\partial^\infty Y_1\cup \partial^\infty Y_2}=\partial^\infty X$.\\
Provided $X_n$ is non-orientable for infinitely many $n\inN$, the surfaces of at least one of the sequences may be chosen to be non-orientable. 
\item[2)] There is a sequence $\seq{Y}{n}$ in $\mathcal{S}\p{c-2}$ converging to some $Y\in\mathcal{M}$ such that $X$ is isometric to $Y$ or a metric 2-point identification of it. Furthermore we find a corresponding isometry or projection map $p$ such that $p\p{\partial^\infty Y}=\partial^\infty X$.
\item[3)] There is a sequence $\seq{Y}{n}$ in $\mathcal{S}\p{c-1}$ converging to some $Y\in\mathcal{M}$ such that $X$ is isometric to $Y$. Furthermore we find a corresponding isometry $p$ such that $p\p{\partial^\infty Y}=\partial^\infty X$.
\end{itemize}
If $X_n$ is orientable for infinitely many $n\inN$, then always one of the first two cases applies and the surfaces of the corresponding sequences may be chosen to be orientable.
\end{lem}
\noindent
The final result of this section refines a statement of the \hyperlink{Main Theorem}{Main Theorem}:
\begin{trm}\label{trm_main_first}
Let $\seq{X}{n}$ be a sequence in $\mathcal{S}\p{c}$ and $X\in\mathcal{M}$ with $\conv{X_n}{X}$. Then there are $k,k_0\inN_0$ and a space $Y\in\mathcal{G}\p{c_0}$, where $c_0\le c+k_0-2k$, such that the following statements apply:
\begin{itemize}
\item[1)] $X$ can be obtained by a successive application of $k$ metric 2-point identifications to $Y$ such that $k_0$ of them are boundary identifications.
\item[2)] If $X_n$ is orientable for infinitely many $n\inN$, then the maximal cyclic subsets of $Y$ are orientable.
\item[3)] If $X_n$ is non-orientable for infinitely many $n\inN$ and the maximal cyclic subsets of $Y$ are orientable, then $c_0<c$. 
\end{itemize}
\end{trm}
\begin{proof}
First we add a statement to the claim: There is a choice $p_1,\ldots,p_k$ of the corresponding projection maps such that $\partial^\infty X\subset \p{p_k\circ\ldots\circ p_0}\p{\partial Y}$ where $p_0\coloneqq id_Y$.\\
The proof proceeds by induction over the connectivity number:\\
In the case $c=0$ the claim directly follows by Theorem \ref{trm_lim_closed}.\\
We consider the case $c>0$. Moreover we assume that the claim is true if the connectivity number is less than $c$. If the sequence $\seq{X}{n}$ contains infinitely many closed surfaces or is regular, the claim follows from Theorem \ref{trm_lim_disc}, Theorem \ref{trm_lim_closed} and Proposition \ref{prop_reg_lim}. Hence we may assume that one of the cases in Lemma \ref{lem_arcs} or Lemma \ref{lem_Jor} applies.\\
We note that the surfaces of the sequences occurring there have a connectivity number less than $c$. Therefore an application of the induction hypothesis yields the claim.  
\end{proof}
\noindent
Finally we are able to prove the local description of the limit spaces:
\begin{proof}[Proof of Theorem \ref{trm_loc}] Due to Theorem \ref{trm_main_first} there is some generalized cactoid $Y$ such that $X$ is homeomorphic to a topological quotient of $Y$ whose underlying equivalence relation identifies only finitely many points. We denote the number of maximal cyclic subsets in $Y$ not being homeomorphic to the 2-sphere or the 2-disc by $k$.\\
First we show the following claim: Every $y\in Y$ admits an open neighborhood being homeomorphic to an open subset of some Peano space whose maximal cyclic subsets are homeomorphic to the 2-sphere or the 2-disc.\\
In the case $k=0$ the claim follows directly.\\
If $k=1$, then we may assume that $y$ lies in the maximal cyclic subset $T\subset Y$ not being homeomorphic to the 2-sphere or the 2-disc. There is a neighborhood $D$ of $y$ in $T$ being homeomorphic to the 2-disc. We denote the union of the connected components of $Y\setminus T$ whose closures intersect $D$ by $A$. It follows that $Z\coloneqq D\cup A$ is a Peano space whose maximal cyclic subsets are homeomorphic to the 2-sphere or the 2-disc.\\
Moreover there is an open neighborhood $V$ of $y$ in $T$ being contained in $D$. We denote the union of the connected components of $Y\setminus T$ whose closures intersect $V$ by $B$. Then $V\cup B$ is an open neighborhood of $y$ in $Y$ and $Z$. This yields the claim.\\
If $k\ge 2$, then $Y$ is a wedge sum of Peano spaces satisfying the following property: All maximal cyclic subsets are compact surfaces and less than $k$ of them are not homeomorphic to the 2-sphere or the 2-disc.\\
Provided both spaces locally look like Peano spaces whose maximal cyclic subsets are homeomorphic to the 2-sphere or the 2-disc, the same applies to $Y$. Hence the claim follows by induction.\\
By the claim $X$ locally looks like a successive wedge sum of Peano spaces whose maximal cyclic subsets are homeomorphic to the 2-sphere or the 2-disc. We note that such a successive wedge sum is then also a Peano space of this kind.
\end{proof}
\section{Approximation of Generalized Cactoids}\label{sec_appr}
In this chapter we show that the second statement of the \hyperlink{Main Theorem}{Main Theorem} implies the first. For this we successively reduce the complexity of the problem.
\subsection{Approximation by Surface Gluings}
The goal of this subsection is to approximate generalized cactoids by suitable spaces in $\mathcal{W}_0$. Our construction extends over the next three results: \begin{lem}
Let $X\in\mathcal{G}\p{c,b}$ and $\seq{T}{k}$ be an enumeration of its maximal cyclic subsets. Then $X$ can be obtained as the limit of compact length spaces $\seq{X}{n}$ satisfying the following properties: 
\begin{itemize}
\item[1)] $X_n$ has only finitely many maximal cyclic subsets. Moreover the maximal cyclic subsets of $X_n$ are in isometric one-to-one correspondence with $\cp{T_k}_{k=1}^n$.
\item[2)] After passing to a subsequence, we may assume the existence of a pre-boundary $\partial^\ast X_n\subset X_n$ with $b$ connected components such that $\conv{\partial^\ast X_n}{\partial X}$. 
\end{itemize}
\end{lem}
\begin{proof}
We extend the author's proof in \cite[p. 9]{Dot24}: First we define an equivalence relation $\sim$ on $X$ as follows: $x\sim y$ if and only if $x$ and $y$ lie in the same connected component of $\cup_{m=n+1}^{n+k}T_m$. Moreover we denote the corresponding metric quotient by $X_{n,k}$ and the corresponding projection map by $p_{n,k}$. Then we may assume that there is a space $X_n\in\mathcal{M}$ and a map $p_n\colon X\to X_n$ such that $\conv{X_{n,k}}{X_n}$ and $\conv{p_{n,k}}{p_n}$ uniformly.\\
From the reference we already know the following: The maximal cyclic subsets of $X_n$ are in isometric one-to-one correspondence with $\cp{T_k}_{k=1}^n$ via the map $p_n$. Further we have $\conv{X_n}{X}$.\\ 
Now we show that $p_n\p{\partial X}$ is a pre-boundary for infinitely many $n\inN$: Let $C$ be a connected component of $\partial X$. Since $p_n$ is continuous, $p_n(C)$ is a subcontinuum of $X_n$.\\
Further let $m\in\cp{1,\ldots,n}$ and $x_1,x_2\in p_n(C)\cap p_n\p{T_m}$ with $x_1\neq x_2$. Then there are $c_i\in C$ and $t_i\in T$ such that $p_n\p{c_i}=p_n\p{t_i}=x_i$. Moreover we choose a geodesic $\gamma_i\subset X$ between $c_i$ and $t_i$ and derive $p_n\p{\gamma_i}=\cp{x_i}$. Since $x_1\neq x_2$, the geodesics do not intersect. From Lemma \ref{lem_bound_cact} it follows that $C$ is arcwise connected and we find a non-degenerate arc $\alpha\subset C$ connecting $c_1$ and $c_2$. Because $\gamma_1$ and $\gamma_2$ do not intersect, we may assume that $\alpha$ intersects $\gamma_1\cup\gamma_2$ only twice. We derive that $\gamma\coloneqq \gamma_1\cup\alpha\cup\gamma_2$ is an arc and Lemma \ref{lem_Peano} implies $\gamma\subset T$. Finally we conclude $x_i\in p_n\p{T\cap C}$.\\
The observation above implies the following: If the subset $p_n(C)\cap p_n\p{T_m}$ is non-degenerate, then it equals $p_n\p{C\cap T_m}$. Because $C$ is admissible and $p_n$ is an isometry on $T_m$, we deduce that $p_n(C)$ is admissible.\\
Due to the fact that $C$ contains a boundary component of some maximal cyclic subset, we may assume that the same applies to $p_n(C)$. Moreover $p_n\p{\partial X}$ covers the boundary components of the maximal cyclic subsets as $\partial X$ does.\\
The map $p_n$ is 1-lipschitz. After passing to a subsequence, we hence may assume $\seq{p}{n}$ to be convergent. We denote its limit by $p$ and it follows that $p$ is an isometry. Further we may assume the sequence $\p{p_n\p{\partial X}}_{n\inN}$ to be convergent. Since $p\p{\partial X}=\partial X$, we have $\conv{p_n\p{\partial X}}{\partial X}$. Therefore we may assume that $\partial X$ has as many connected components as $p_n\p{\partial X}$. We finally deduce that $p_n\p{\partial X}$ is a pre-boundary with $b$ connected components.
\end{proof}
\begin{lem}
Let $X$ be a geodesic generalized cactoid having only finitely many maximal cyclic subsets. Further let $\partial^\ast X\subset X$ be a pre-boundary with $b$ connected components. Then $X$ can be obtained as the Hausdorff limit of compact subsets $\seq{X}{n}$ satisfying the following properties: 
\begin{itemize}
\item[1)] We have $X_n\in\mathcal{W}$ and the maximal cyclic subsets of $X_n$ are equal to those of $X$.
\item[2)] There is a pre-boundary $\partial^\ast X_n\subset X_n$ with $b$ connected components. 
\item[3)] The sequence $\seq{\partial^\ast X}{n}$ Hausdorff converges to $\partial^\ast X$.
\end{itemize}
\end{lem}
\noindent
\begin{proof} We extend the author's proof in \cite[p. 9]{Dot24}: Let $\varepsilon$ be the minimum of the diameters of the maximal cyclic subsets of $X$. For every maximal cyclic subset $T\subset X$ we remove the connected components of $X\setminus T$ whose diameters are less than $\nicefrac{\varepsilon}{n}$. We denote the constructed subset by $Y_n$.\\
The subset $Y_n$ is a successive metric wedge sum of its maximal cyclic subsets and compact metric trees $D_1,\ldots, D_k$. For every metric tree $D_i$ there is a finite metric tree $F_i\subset D_i$ whose Hausdorff distance to $D_i$ is less than $\nicefrac{\varepsilon}{n}$ \comp{p. 267}{BBI01}. In particular we may assume that $F_i$ intersects the same maximal cyclic subsets as $D_i$. Next we define $X_n$ as the union of the maximal cyclic subsets of $Y_n$ and the finite metric trees. Moreover we set $\partial^\ast X_n\coloneqq \partial^\ast X\cap X_n$.\\
By construction we have $X_n\in\mathcal{W}$. Further the maximal cyclic subsets of $X_n$ are equal to those of $X$ and $\partial^\ast X_n$ is a pre-boundary of $X_n$ with $b$ connected components. Finally the sequences $\seq{X}{n}$ and $\seq{\partial^\ast X}{n}$ Hausdorff converge to $X$ and $\partial^\ast X$.
\end{proof}
\noindent
\begin{lem}
Let $X$ be a generalized cactoid in $\mathcal{W}$. Further let $\partial^\ast X\subset X$ be a pre-boundary with $b$ connected components. Then there is a sequence $\seq{X}{n}$ in $\mathcal{W}_0$ satisfying the following properties: 
\begin{itemize}
\item[1)] The maximal cyclic subsets of $X_n$ are in isometric one-to-one correspondence with the maximal cyclic subsets of $X$ and finitely many length spaces being homeomorphic to the 2-sphere or the 2-disc.
\item[2)] $X_n$ has $b$ boundary components. 
\item[3)] There is an $\varepsilon_n$-isometry $f_n\colon X_n\to X$ such that $f_n\p{\partial X_n}=\partial^\ast X$ and $\conv{\varepsilon_n}{0}$.
\end{itemize}
\end{lem} 
\begin{proof}
The space $X$ is a successive metric wedge sum of its maximal cyclic subsets and compact intervals. In particular we may assume that the wedge points do not lie in the interior of the intervals and that every interval whose interior intersects $\partial^\ast X$ lies in $\partial^\ast X$.\\
We consider the following construction: Let $e$ be one of the wedged intervals. Then there is a $\nicefrac{1}{n}$-isometry $f\colon D\to e$ satisfying the following properties: The preimages of the endpoints of $e$ contain exactly one point. If $e$ lies in $\partial^\ast X$, then $D$ is a length space being homeomorphic to the 2-disc and $f\p{\partial D}=e$. Otherwise $D$ is a length space being homeomorphic to the 2-sphere.\\
Now we remove the interior of $e$ from $X$ and paste $D$ along $f$. This yields a compact length space $Y$. Especially the maximal cyclic subsets of $Y$ are in isometric one-to-one correspondence with the set consisting of $D$ and the maximal cyclic subsets of $X$.\\
If $e$ lies in $\partial^\ast X$, then we set $\partial^\ast Y\coloneqq\p{\partial^\ast X\setminus e}\cup\partial D$. Otherwise we set $\partial^\ast Y\coloneqq \partial^\ast X$. We note that $\partial^\ast Y$ is a pre-boundary of $Y$ with $b$ connected components. Provided $Y$ is a successive metric wedge sum of its maximal cyclic subsets, we have $\partial^\ast Y=\partial Y$.\\ 
Furthermore the map $f$ naturally induces $\nicefrac{1}{n}$-isometry $g\colon Y\to X_n$ with $g\p{\partial^\ast Y}=\partial^\ast X$.\\
We successively repeat this construction until there is no wedged interval left and denote the constructed space by $X_n$.\\ 
The space $X_n$ is a successive metric wedge sum of its maximal cyclic subsets. For every wedge point $p\in X_n$ which lies in more than one maximal cyclic subset there is a $\nicefrac{1}{n}$-isometry $f\colon D\to \cp{p}$ satisfying the following properties: If $p$ lies in $\partial X_n$, then $D$ is a length space being homeomorphic to the 2-disc. Otherwise $D$ is a length space being homeomorphic to the 2-sphere.\\
Using a similar construction as above, we finally may assume that $X_n\in\mathcal{W}_0$.
\end{proof}
\noindent
We note that a sequence of  $\varepsilon_n$-isometries between converging spaces has a convergent subsequence provided $\conv{\varepsilon_n}{0}$. In particular its limit is an isometry between the limit spaces. Morever the boundary of a generalized cactoid is invariant under self-isometries. Combining Proposition \ref{prop_almost_isom} and the last three results, we hence get the desired approximating sequence:
\begin{cor}\label{cor_appr_G}
Let $X\in\mathcal{G}\p{c,b}$ and $\seq{T}{k}$ be an enumeration of its maximal cyclic subsets. Then $X$ can be obtained as the limit of spaces $\seq{X}{n}$ in $\mathcal{W}_0$ satisfying the following properties:
\begin{itemize}
\item[1)] The maximal cyclic subsets of $X_n$ are in isometric one-to-one correspondence with $\cp{T_k}_{k=1}^{n}$ and finitely many length spaces being homeomorphic to the 2-sphere or the 2-disc.
\item[2)] After passing to a subsequence, we may assume that $X_n$ has $b$ boundary components and $\conv{\partial X_n}{\partial X}$. 
\end{itemize}
\end{cor}
\noindent 
\subsection{Elementary Surface Gluings}
Now we provide useful tools concerning the approximation of elementary surface gluings. We start with wedge sums: 
\begin{lem}\label{lem_wedge}
Let $S_1\in\mathcal{S}\p{c_1,b_1}$, $S_2\in\mathcal{S}\p{c_2,b_2}$ and $X$ be a metric wedge sum of $S_1$ and $S_2$. Then the following statements apply:
\begin{itemize}
\item[1)] There is a sequence $\seq{X}{n}$ in $\mathcal{S}\p{c_1+c_2,b_1+b_2}$ and an $\varepsilon_n$-isometry $f_n\colon X_n\to X$ such that $f_n\p{\partial X_n}=\partial X$ and $\conv{\varepsilon_n}{0}$. 
\item[2)] If the wedge point is contained in $\partial S_1\cap\partial S_2$, then there is a sequence $\seq{X}{n}$ in $\mathcal{S}\p{c_1+c_2,b_1+b_2-1}$ and an $\varepsilon_n$-isometry $f_n\colon X_n\to X$ such that $f_n\p{\partial X_n}=\partial X$ and $\conv{\varepsilon_n}{0}$. 
\end{itemize}
If $S_1$ and $S_2$ are orientable, then the surfaces of the sequence may be chosen to be orientable. Provided at least one of the wedged surfaces is non-orientable, the surfaces of the sequence may be chosen to be non-orientable.
\end{lem}
\begin{proof}
1) We may assume the wedge points not to lie in $\partial S_1\cup\partial S_2$. In \cite[p. 14]{Dot24} the author already showed the statement for closed surfaces. The corresponding proof does not depend on the fact that the wedged surfaces are closed and it gives rise to a proof for the general case.\\ 
2) Let $a_1$ and $a_2$ be the intersecting boundary components of the surfaces and $p$ be the wedge point. We choose an arc $\gamma_{i,n}\subset a_i$ containing $p$ in its interior. In particular we may assume that the arc is a geodesic of length $\nicefrac{1}{n}$ such that $p$ is its midpoint. Next we define $X_n$ as the metric gluing of $X$ along $\gamma_{1,n}$ and $\gamma_{2,n}$. Moreover we denote the corresponding projection map by $g_n$ and find a map $f_n\colon X_n\to X$ such that $f_n\circ g_n$ is the identity map on $X_n\setminus\p{\gamma_{1,n}\cup\gamma_{2,n}}$ and $\p{f_n\circ g_n}\p{\gamma_{1,n}}=\gamma_{1,n}\cup \gamma_{2,n}$. Finally we deduce that the space $X_n$ and the map $f_n$ satisfy the desired properties.
\end{proof}
\noindent
Using similar arguments as above, we derive the following result concerning metric 2-point identifications: \begin{lem}\label{lem_2point}
Let $S$ be a space in $\mathcal{S}(c)$ and $X$ be a metric 2-point identification of $S$. Further let $p$ be a corresponding projection map. Then the following statements apply: 
\begin{itemize}
\item[1)] There is sequence $\seq{X}{n}$ in $\mathcal{S}\p{c+2}$ and an $\varepsilon_n$-isometry $f_n\colon X_n \to X$ with
$f_n\p{\partial X_n}=p\p{\partial S}$ and $\conv{\varepsilon_n}{0}$. 
\item[2)] If $p$ is a boundary identification, then there is a sequence $\seq{X}{n}$ in $\mathcal{S}\p{c+1}$ and an $\varepsilon_n$-isometry $f_n\colon X_n \to X$ with $f_n\p{\partial X_n}=p\p{\partial S}$ and $\conv{\varepsilon_n}{0}$.  
\end{itemize}
The surfaces of the sequence may be chosen to be non-orientable. If $S$ is orientable, then the surfaces of the sequence may be chosen to be orientable. 
\end{lem}
\subsection{Gluings of Generalized Cactoids}
In this subsection we approximate spaces that can be obtained by a successive application of $k$ metric 2-point identifications to some generalized cactoid.\\
Using an induction, Corollary \ref{cor_appr_G} and Lemma \ref{lem_wedge} yield the case $k=0$:
\begin{cor}\label{cor_appr_G_II}
Let $X\in\mathcal{G}\left(c,b\right)$. Then $X$ can be obtained as the limit of spaces $\seq{X}{n}$ in $\mathcal{S}\left(c,b\right)$. Moreover the following statements apply: 
\begin{itemize}
\item[1)] The sequence may be chosen such that $\conv{\partial X_n}{\partial X}$.
\item[2)] If the maximal cyclic subsets of $X$ are orientable, then the surfaces of the sequence may be chosen to be orientable.
\item[3)] If there is a non-orientable maximal cyclic subset in $X$, then the surfaces of the sequence may be chosen to be non-orientable. 
\end{itemize}\end{cor}
\noindent
Now we show the general case:
\begin{lem}
Let $k,k_0\inN_0$ and $X\in\mathcal{G}(c)$. Further let $Y$ be a space that can be obtained by a successive application of $k$ metric 2-point identifications to $X$ such that $k_0$ of them are boundary identifications. Then $Y$ can be obtained as the limit of spaces $\seq{Y}{n}$ in $\mathcal{S}\p{c-k_0+2k}$. Moreover the following statements apply:
\begin{itemize}
\item[1)] If the maximal cyclic subsets of $X$ are orientable, then the surfaces of the sequence may be chosen to be orientable. 
\item[2)] If there is a non-orientable maximal cyclic subset in $X$ or $k>0$, then the surfaces of the sequence may be chosen to be non-orientable.
\end{itemize}
\end{lem}
\begin{proof}
First we add a statement to the claim: There is a choice $p_1,\ldots,p_k$ of the corresponding projection maps such that $\conv{\partial Y_n}{\p{p_k\circ\ldots\circ p_0}\p{\partial X}$ where $p_0\coloneqq id_X$}.\\
The proof of the claim proceeds by induction over $k$:\\
In the case $k=0$ the claim directly follows by Corollary \ref{cor_appr_G_II}.\\
Now we consider the case $k>0$. Moreover we assume that the claim is true if the number of identifications is less than $k$. Let $p_1,\ldots, p_k$ be a choice of the corresponding projection maps. We set $Z\coloneqq \p{p_{k-1}\circ\ldots\circ p_0}(X)$ and denote the number of boundary identifications in $\cp{p_{k-1},\ldots,p_1}$ by $\Tilde{k}_0$. Then $Z$ is a space that can be obtained by a successive application of $k-1$ metric 2-point identifications to $X$ such that $\Tilde{k}_0$ of them are boundary identifications. Hence we can apply the induction hypothesis and derive a corresponding sequence $\seq{Z}{n}$ in $\mathcal{S}\p{c-\Tilde{k}_0+2\p{k-1}}$. In particular we may assume that $\conv{\partial Z_n}{\p{p_{k-1}\circ\ldots\circ p_0}\p{\partial X}}$.\\
Let $z_1,z_2\in Z$ be distinct  points with $p_k\p{z_1}=p_k\p{z_2}$. Then there is a sequence $\p{z_{i,n}}_{n\inN}$ with $z_{i,n}\in Z_n$ and $z_{1,n}\neq z_{2,n}$ such that $\conv{z_{i,n}}{z_i}$. Provided $p_k$ is a boundary identification, we may assume $z_{i,n}\in\partial Z_n$. Further we define $W_n$ as the metric gluing of $Z_n$ along $z_{1,n}$ and $z_{2,n}$. We denote the corresponding projection map by $p_{k,n}$.\\
By construction $p_{k,n}$ is a boundary identification if $p_k$ is. Moreover it follows $\conv{W_n}{Y}$ and we may assume that $\p{p_{k,n}}_{n\inN}$ converges to some map $q_k$. We note that $q_k,p_{k-1},\ldots,p_1$ is also a possible choice of the projection maps corresponding to the construction of $Y$. Hence we may assume that $\conv{p_{k,n}\p{\partial Z_n}}{\p{p_k\circ\ldots\circ p_0}\p{\partial X}}$.\\
Finally we apply Lemma \ref{lem_2point} to $W_n$ and denote the corresponding sequence of surfaces by $\p{Y_{n,m}}_{m\inN}$ and the corresponding sequence of almost isometries by  $\p{f_{n,m}}_{m\inN}$. Then we have $f_{n,m}\p{\partial Y_{n,m}}=p_{k,n}\p{\partial Z_n}$. Choosing a diagonal sequence, we may assume that $\conv{Y_n\coloneqq Y_{n,n}}{Y}$ and that $\p{f_{n,n}}_{n\inN}$ converges to an isometry $f$. We note that $\p{f^{-1}\circ p_k}, p_{k-1},\ldots, p_1$ is also a possible choice of the projection maps corresponding to the construction of $Y$. Hence we finally may assume that $\conv{\partial Y_n}{\p{p_k\circ\ldots\circ p_0}\p{\partial X}}$.
\end{proof}
\noindent
Using a metric wedge sum with a vanishing sequence of length spaces all being homeomorphic to the 2-disc or all being homeomorphic to the real projective plane, we derive a further corollary of Lemma \ref{lem_wedge}:
\begin{cor}
Let $X\in\mathcal{S}(c)$. Then $X$ can be obtained as the limit of spaces in $\mathcal{S}\p{c+1}$. Moreover the following statements apply:
\begin{itemize}
\item[1)] If $S$ is orientable, then the surfaces of the sequence may be chosen to be orientable.
\item[2)] The surfaces of the sequence may be chosen to be non-orientable.
\end{itemize}
\end{cor}
\noindent
The final result of this chapter is a direct consequence of the last two results and it refines the remaining statement of the \hyperlink{Main Theorem}{Main Theorem}:
\begin{trm}\label{trm_main_sec}
Let $c,k,k_0\inN_0$ and $X\in\mathcal{G}\p{c_0}$ where $c_0\le c+k_0-2k$. Further let $Y$ be a space that can be obtained by a successive application of $k$ metric 2-point identifications to $X$ such that $k_0$ of them are boundary identifications. Then $Y$ can be obtained as the limit of spaces $\seq{Y}{n}$ in $\mathcal{S}(c)$. Moreover the following statements apply: \begin{itemize}
\item[1)] If the maximal cyclic subsets of $X$ are orientable, then the surfaces of the sequence may be chosen to be orientable.
\item[2)] If there is a non-orientable maximal cyclic subset in $X$ or $c_0<c$, then the surfaces of the sequence may be chosen to be non-orientable.
\end{itemize}
\end{trm}
\noindent
Finally we prove the third statement of Theorem \ref{trm_topo}:
\begin{proof}[Proof of Theorem \ref{trm_topo} (Part II)] If $Y$ is a space that can be obtained by a successive application of metric 2-point identifications to some geodesic generalized cactoid, then the second statement of Theorem \ref{trm_topo} and Theorem \ref{trm_main_sec} imply that $Y$ is locally simply connected. Due to Theorem \ref{trm_main_first} the space $X$ can be obtained in this way. Hence Proposition \ref{prop_funda_form} and Proposition \ref{prop_funda_form_II} finally yield the desired fundamental group formula for $X$.  
\end{proof}
\section*{Acknowledgments}
The author thanks his PhD advisor Alexander Lytchak for great support.
\bibliography{sources}{}
\bibliographystyle{abbrv}
\footnotesize{\textsc{Institute of Algebra and Geometry, Karlsruhe Institute of Technology, Germany.}\\
\emph{E-mail}: \url{tobias.dott@kit.edu}}
\end{document}